\documentclass[11pt, reqno]{amsart}
\usepackage[utf8]{inputenc}
\usepackage{amsmath}
\usepackage{amsthm}
\usepackage{amsfonts}
\usepackage{amssymb}
\usepackage{mathrsfs, mathtools, mathdots}
\usepackage{diagbox}
\usepackage{cancel}

\newtheorem{theorem}{Theorem}[section]
\newtheorem{lemma}{Lemma}[section]
\newtheorem{proposition}{Proposition}[section]
\newtheorem{corollary}{Corollary}[section]

\newtheorem{remark}{Remark}
\usepackage{hyperref}
\usepackage[top=30mm,bottom=30mm,left=30mm,right=30mm]{geometry}
\newcommand{\KK}[1]{\mathcal{O}_{#1}}

\author{Saunak Bhattacharjee}
\address{School of Science, University of New South Wales, Canberra, ACT, Australia.}
\email{saunak.bhattacharjee@unsw.edu.au}
\title[Larger sieve with height function and uniform bounds for integral points]{Larger sieve with height function and uniform bounds for integral points on curves over number fields}
\begin{document}
\maketitle
\begin{abstract}
    Let $A \subseteq \KK{K}$ be a set of algebraic integers of height up to $H$ such that $|A \mod{\mathfrak{p}}|\leq \alpha |\KK{K}/\mathfrak{p}|$ for every prime ideal $\mathfrak{p}$ with $N\mathfrak{p}>c$ for some $\alpha \in (0,1)$. It follows from a larger sieve due to Ellenberg, Elsholtz, Hall and Kowalski that $|A| \ll_{K,c,\alpha}H^{2\alpha}$. In this paper, we improve on this larger sieve bound by showing that $|A|\ll_{K,c,\alpha}H^{\alpha}(\log H)^r$. We also obtain a two-dimensional larger sieve of  Helfgott and Venkatesh type over $\KK{K} \times \KK{K}$ and apply it to produce a Bombieri-Pila type bound over $\KK{K}$.

\end{abstract}
\author{}

\date{}

\maketitle

\section{Introduction}

Let $A \subseteq [0,N]$ be a set of natural numbers. Suppose that for every prime $p>c$,
\[
|A \bmod p| \le \alpha p
\]
for some $\alpha \in (0,1)$. In other words, the set $A$ is ill-distributed modulo $p$. Gallagher's larger sieve \cite{Gallagher} shows that occupying only a small proportion of residue classes modulo $p$ forces the set to be small. In particular,
\begin{equation}\label{0}
    |A|\ll_{c,\alpha}N^{\alpha}.
\end{equation}

\medskip

\noindent
In this paper we study a number field analogue of this problem. Let $K$ be a number field and let $\KK{K}$ denote its ring of integers. Consider the set
\[
A=\{a\in \KK{K}: H(a)\leq H\},
\]
where $H(a)$ denotes the multiplicative height of $a$, defined by 
\begin{equation}\label{HD}
    H(a)=\prod_{v\in M_K}\max\bigl(1,|a|_v^{n_v}\bigr),
\end{equation}
where $M_K$ denotes the set of places of $K$ \cite[p. 206]{Silverman}. See (\ref{hdef}) and Section \ref{Backgrond} for a detailed discussion of the height function.

\medskip

\noindent
Suppose that
\[
|A \bmod \mathfrak{p}| \le \alpha |\KK{K}/\mathfrak{p}|
\]
for every prime ideal $\mathfrak{p}$ with $N\mathfrak{p}>c$ and some $\alpha \in (0,1)$. How small can $|A|$ be?

\medskip
\noindent
Ellenberg, Elsholtz, Hall and Kowalski \cite{Ellenberg} and Zywina \cite{Zywina} introduced a larger sieve over an arbitrary number field $K$ with the same height function (\ref{HD}) as above. From their larger sieve \cite[Proposition 17]{Ellenberg}, it follows that
\begin{equation}\label{1}
    |A| \ll_{K,c,\alpha}H^{2\alpha}.
\end{equation}

\noindent
Note that the exponent of $H$ is twice that in (\ref{0}), whereas the corresponding exponent for rational integers is only $\alpha$. This extra $2$-factor appears in the exponent because of the multiplicative nature of the height function, i.e. in \cite{Ellenberg} the authors use the following well-known property of the height function which contributes.
\begin{equation}\label{2}
    H(\alpha+\beta)\leq 2^{[K:\mathbb{Q}]}H(\alpha)H(\beta), \text{\,\,\,\,\,for every\,\,} \alpha, \beta \in K.
\end{equation}

\noindent
If we fix $K=\mathbb{Q}$, then for every rational integer $\alpha, \beta \in \mathbb{Z} $ we have
$$H(\alpha+\beta)\leq H(\alpha)+H(\beta).$$
Consequently, we get a bound similar to the larger sieve bound for natural numbers (\ref{0})  $$|A|\ll_{c,\alpha} H^{\alpha}.$$ 
\noindent
It is therefore natural to ask whether the bound in (\ref{1}) is sharp. 
In  \cite[Section 6]{Shao}, Shao raises this question about the sharpness of the larger sieve bound in the setting of Ellenberg, Elsholtz, Hall and Kowalski \cite{Ellenberg}. Moreover, in  \cite[Section 4]{Shao}, he shows that the inverse sieve conjecture implies improved bound on the larger sieve. Thus, understanding the sharpness of the larger sieve over number fields is relevant to formulating an improved larger sieve conjecture over number fields.

\medskip
\noindent
In Section~\ref{Larger Sieve}, we prove the following version of the larger sieve over the ring of integers $\mathcal{O}_K$ to address the question about the sharpness of the larger sieve bound over number fields.
\begin{theorem}\label{L.S}
Let $K/\mathbb{Q}$ be a number field, let $H > 0$ be a constant, and let $A$ be a finite set of elements of $\KK{K}$ such that 
\[
    H(a) \leq H \quad \text{for all } a \in A,
\]
where $H$ denotes the multiplicative height (\ref{HD}) on $K$.\\  
\noindent
Let $P$ be a finite set of prime ideals in the ring of integers $\KK{K}$.  
If  $$|A \bmod{\mathfrak{p}}|\leq \nu(\mathfrak{p})$$ for all $\mathfrak{p} \in P$, then 
\[
    |A| \ll_K 
    \Bigg(\frac{
        \displaystyle \sum_{\mathfrak{p} \in P} \log N\mathfrak{p} 
        - \log\!\big(2C_K H \big)
    }{
        \displaystyle \sum_{\mathfrak{p} \in P} 
        \frac{\log N\mathfrak{p}}{\nu(\mathfrak{p})}
        - \log\!\big( 2C_K H \big)
    }\Bigg)(\log H)^r,
\]
where $C_K$ is a constant only depending on the number field $K$, $r$ is the rank of the unit group $\KK{K}^*$ and provided that the denominator in this expression is positive.
\end{theorem}

\begin{remark}
    In particular, from the proof of Theorem~\ref{L.S} it follows that the constant $C_K$, appearing inside the logarithm $\log\!\big(2C_K H \big)$ in the statement of Theorem~\ref{L.S} depends on the degree $[K:\mathbb{Q}]$, the rank $r$ of the unit group, and the quantity $\max_{1\le i\le n,\,1\le j\le r}\log|\sigma_i(\epsilon_j)|$, where $\epsilon_1,\dots,\epsilon_r$ are fundamental units of $\KK{K}^*$ and $\sigma_i:K\to\mathbb{C}$ are the embeddings.
\end{remark}

\noindent
As an immediate consequence of Theorem \ref{L.S}, we obtain the following corollary, which improves upon the bound (\ref{1}) obtained using the larger sieve of Ellenberg, Elsholtz, Hall, and Kowalski.
\begin{corollary}\label{cor}
    Let $A=\{a\in \KK{K}:H(a)\leq H\}.$ If for every prime ideal $\mathfrak{p}
$ with $N\mathfrak{p}>c$, $$|A \bmod\mathfrak{p}|\leq \alpha |\KK{k}/\mathfrak{p}|, \text{\,\,\,\,\,for some } \alpha \in (0,1).$$  Then 
$$|A|\ll_{K,c,\alpha}H^{\alpha}(\log H)^r.$$
\end{corollary}
\noindent
This power saving is achieved by improving on the trivial height inequality (\ref{2}). More precisely, we prove an additive version (Proposition \ref{I1}) of the height inequality for a large subset of algebraic integers up to a given height.

\medskip
\noindent
Further, in Section $\ref{2-D Larger}$, we study the analogous problem of two-dimensional ill-distributed sets of algebraic integers. In \cite{H-V}, Helfgott and Venkatesh introduced a two-dimensional larger sieve \cite[Proposition $3.1$]{H-V} to study two-dimensional ill-distributed sets of rational integers and resolved the two-dimensional inverse sieve conjecture (see \cite{Walsh}, \cite{Shao}, \cite{Green-Harper} for a detailed discussion on the inverse sieve conjectures). Moreover, they recovered the Bombieri-Pila  bound \cite{B-P} using \cite[Proposition $3.1$]{H-V}.

\medskip
\noindent
In \cite{Walsh}, Walsh settled the higher dimensional inverse sieve conjectures. Recently, in \cite{Sasyk IS}, Menconi, Paredes and Sasyk extended and generalized the results of Walsh \cite{Walsh} to global fields.
Note that, in \cite[Section 3.2]{Sasyk IS}, the authors use the trivial height inequality (\ref{2}) to obtain analogues of Lemma 3.1 and Lemma 3.2 of \cite{Walsh}. Although the asymptotic bounds in the main results of \cite{Sasyk IS} remain unaffected by using the improved height inequality (Proposition \ref{I1}), one can indeed get improved constants $C_1$ and $C_2$ in Lemma 3.1 and 3.2 of \cite{Sasyk IS} respectively.

\medskip
\noindent
In a separate paper \cite{Sasyk BP} Paredes and Sasyk produced a Bombieri-Pila type bound over global fields generalizing the methods of Heath-Brown \cite{Heath-Brown} as well as extending the works of Salberger \cite{Salberger}, Walsh \cite{Walsh 2}, and Castryck, Cluckers, Dittmann and Nguyen \cite{others}. In this paper, we study the same problems over number fields but we do so by extending the methods of Helfgott and Venkatesh to number fields. 

\medskip
\noindent
It turns out that when one tries to generalize the two-dimensional larger sieve of Helfgott and Venkatesh to number fields and then to produce the Bombieri-Pila bound, a similar obstacle arises due to the multiplicative nature of the height function. 
In Section \ref{2-D Larger}, we resolve this issue by proving an improved height inequality over determinants (Proposition \ref{H2}) and using this we prove a two dimensional larger sieve (Theorem \ref{2-D Larger Sieve}) of Helfgott and Venkatesh type over number fields. Further in Section \ref{Bombieri-Pila}, we apply it (Theorem \ref{2-D Larger Sieve}) to reproduce a Bombieri-Pila type bound (Theorem \ref{B-P}) over number fields .

\medskip
\noindent
In particular, let $\mathcal{W}$ be a finite set of monomials $x^ly^m$ with $0\leq l\leq L$ and $0\leq m\leq M$, where $L$ and $M$ are given. Suppose,
\[|\mathcal{W}|=w, \,\,\,\,\,
d_{\mathcal{W}} = \sum_{f \in \mathcal{W}} \deg(f).
\]
Write $f_1,f_2,...,f_w$ for the elements of $\mathcal{W}$ and by a $\mathcal{W}$-curve we mean an affine algebraic curve described by a single equation $g(x,y)=0$, where $g$ belongs to the linear span of $\mathcal{W}$ over $K$. 

\medskip
\noindent
We prove  the following version of the two-dimensional larger sieve of Helfgott and Venkatesh type over $\KK{K}$.

\begin{theorem}\label{2-D Larger Sieve}
    Let $S \subseteq \{(\alpha,\beta )\in \KK{K}\times \KK{K}\ : H(\alpha), H(\beta)\leq H\}$.  
Suppose that for some fixed $\tau>0$ and some constant $c>0$ the number of residue classes
\[
\{(\alpha,\beta) \bmod \mathfrak{p} : (\alpha,\beta)\in S\}
\]
is at most $\tau |\KK{K}/\mathfrak{p}|$ for every prime ideal $\mathfrak{p}$ with $N\mathfrak{p}>c$.  

\noindent
Then for any $\delta\in(0,1)$ one of the following holds:
\begin{enumerate}
  \item[(a)] there is a $\mathcal{W}$-curve and $S'\subseteq S$ with $|S'|\gg_K\frac{|S|}{(\log H)^{2r}}$ such that the $\mathcal{W}$-curve contains at least $\delta|S'|$ points of $S'$. 
  
  \item[(b)] 
  $$|S| \ll_{K,c,\delta,\mathcal{W}} H^{ \frac{2\tau\, d_{\mathcal{W}}}{w(w-1)} + O_{\tau,\mathcal{W}}(\delta)}.$$
  
  \end{enumerate}
\end{theorem}

\medskip
\noindent
The condition that $S$ occupies at most $\tau |\KK{K}/\mathfrak{p}|$ residue classes modulo $\mathfrak{p}$ is quite strong, even when $\tau$ is large. As noted in \cite[p. 2]{H-V}, this is analogous to a typical two-dimensional subset of integers occupying at most $\alpha p$ residue classes modulo $p$. Such a condition naturally arises when $S$ has a large intersection with a curve of low degree. Hence, Theorem \ref{2-D Larger Sieve} may be viewed as an intermediate step toward establishing a Bombieri--Pila type bound over number fields.

\medskip
\noindent
On the other hand, if $S \subseteq \{(\alpha,\beta )\in \KK{K}\times \KK{K}\ : H(\alpha), H(\beta)\leq H\}$ satisfies a condition of the form
\[
|S \bmod \mathfrak{p}| \le \tau |\KK{K}/\mathfrak{p}|^{m},
\]
where $m>1$, then the extension of the methods of Helfgott and Venkatesh is no longer effective. Instead, the generalisation of Walsh's methods \cite{Walsh} by Menconi, Paredes, and Sasyk in \cite{Sasyk IS} becomes applicable.

\medskip
\noindent
Next, as an application of this two-dimensional larger sieve over $\KK{K}$, we obtain the following Bombieri-Pila type bound over number fields in Section \ref{Bombieri-Pila}.
 \begin{theorem}\label{B-P}
    Let $f(x,y) \in \mathcal{O}_K[x,y]$ be irreducible over $K$ of degree $d$ and $$S=\{(x,y)\in \mathcal{O}_K^2, H(x),H(y)\leq H : f(x,y)=0\}.$$
    Then, $$|S|\ll_{K,d,\epsilon}H^{\frac{1}{d}+\epsilon}.$$
\end{theorem}

\noindent
Note that this bound is similar to \cite[Theorem $1.9$]{Sasyk BP} in the case $n=2$, although the authors use a different height function. Instead of assuming that the height of each coordinate $x$ and $y$ is bounded by $H$, they consider each coordinate to lie in the set $[B]_{\KK{K}}$ \cite[Section 5.2]{Sasyk BP} of algebraic integers satisfying
\[
\max_i |\sigma_i(\,\,\cdot\,\,)| \leq B,\,\, \text{where}\,\,\,\sigma_i:K\to\mathbb{C}\,\,\, \text{are the embeddings.}
\]
Since the quantity $\max_i |\sigma_i(\,\,\cdot\,\,)|$ behaves additively under the addition of algebraic integers, the method of Helfgott and Venkatesh extends more naturally to number fields with the conditions in \cite[Theorem $1.9$]{Sasyk BP}.

\medskip

\noindent
Moreover, the assumption in Theorem \ref{B-P} that each coordinate $x$ and $y$ has height at most $H$ is more general than the condition in \cite[Theorem $1.9$]{Sasyk BP} for $n=2$.

\medskip
\noindent
Before we proceed to the proof of the main results in next sections, let us briefly recall the proof of Gallagher's larger sieve and discuss the main issues that arise in extending it to number fields. 

\medskip
\noindent
We are given that the set $A \subseteq \mathbb{Z} \cap [0,N] $
intersects at most $\alpha p$ residue classes modulo $p$ for every prime
$p > c$, where $\alpha \in (0,1)$ and $c>0$ are fixed constants.
Let us consider the product
\[
\Delta \;=\; \prod_{\substack{x,y \in A \\ x \neq y}} (x-y).
\]
\noindent
The main idea of the larger sieve is to bound the quantity $|\Delta|$ in terms of $|A|$ from below and above and then to compare the bounds.

\noindent
For each prime $p>c$, an application of Cauchy-Schwartz inequality gives
\[
v_p(\Delta)
=\{(x,y)\in A\times A: x\equiv y\pmod  p, \,\,\,x\neq y\}\;\ge\;
\frac{|A|^2}{\alpha p} - |A|.
\]

\noindent
Hence,
\[
|\Delta|
\;\ge\;
\prod_{c < p \le |A|}
p^{\,v_p(\Delta)}
\;\ge\;
\prod_{c < p \le |A|}
p^{\,\frac{|A|^2}{\alpha p} - |A|}
\;\ge\;
\exp\!\left(
\frac{|A|^2}{\alpha}
(\log |A| - O(1))
\right).\]

\noindent
On the other hand we have the trivial upper bound 
$$|\Delta|\leq N^{|A|^2}.$$

\noindent
Comparing the upper and lower bounds one arrives at $|A|\ll_{\alpha}N^{\alpha}$.

\medskip
\noindent
Similarly, in the case of algebraic integers, one will be interested in estimating the quantity
$$H(\Delta)=H\big(\prod_{\substack{x,y \in A \\ x \neq y}} (x-y)\big),$$
\noindent
from above and below. The lower bound is obtained by showing that the $\mathfrak{p}$-adic valuations are large. This comes from the larger sieve assumption of $|A \mod{\mathfrak{p}}|\leq \alpha |\KK{K}/\mathfrak{p}|=\alpha N\mathfrak{p}$. In particular, the lower bound  turns out to be of the similar form 
$$H(\Delta)\geq\prod_{c < N\mathfrak{p} \le |A|}
N\mathfrak{p}^{\,\frac{|A|^2}{\alpha N\mathfrak{p}} - |A|}.
$$
\noindent
Whereas the upper bound becomes $H(\Delta)\leq (2^{[K:\mathbb{Q}]}H^2)^{|A|^2},$ using the trivial height inequality (\ref{2}) which is also used in the proof of Proposition 17 of \cite{Ellenberg} and as a result one obtains
$$|A| \ll_{K,\alpha,c}H^{2\alpha}.$$
\noindent
The main ingredient of our method is the construction of a large subset of $A$ on which the height function behaves additively under the addition of algebraic integers. This leads to an improvement of the trivial upper bound for $H(\Delta)$ and consequently reduces the exponent from $2\alpha$ to $\alpha$.

\medskip

\noindent
A similar idea is used in extending the method of Helfgott and Venkatesh \cite{H-V} to number fields. In particular, let $S \subseteq \{(\alpha,\beta )\in \KK{K}\times \KK{K}\ : H(\alpha), H(\beta)\leq H\}$.  
Suppose that for some fixed $\tau>0$ and some constant $c>0$ the number of residue classes
\[
\{(\alpha,\beta) \bmod \mathfrak{p} : (\alpha,\beta)\in S\}
\]
is at most $\tau |\KK{K}/\mathfrak{p}|$ for every prime ideal $\mathfrak{p}$ with $N\mathfrak{p}>c$.  Then, following the same steps as in the proof of Proposition $3.1$ of \cite{H-V} and using the trivial height inequality (\ref{2}), to bound the height of the determinant $\text{det}((f_i(P_j))_{1\leq i,j\leq w}), \,\, \text{with}\,\, P_j \in S$ from above, one ends up with the upper bound in option $(b)$ of Proposition 3.1 of \cite{H-V} as 
$$|S| \ll_{K,c,\delta,\mathcal{W}} H^{ \frac{4\tau\, d_{\mathcal{W}}}{w(w-1)} + O_{\tau,\mathcal{W}}(\delta)}, \,\,\,\,\,\, \text{for any} \,\,\, \delta \in (0,1)$$

\noindent
Now, as we have a $4$-factor in the exponent instead of $2$, this two-dimensional larger sieve bound fails to reproduce the Bombieri-Pila bound over number fields using the methods in \cite{H-V}.
We resolve this issue in the two-dimensional setting by constructing a large subset of two-dimensional set of algebraic integers on which the height function behaves additively with respect to the relevant determinants arising from the points in the set. This leads to an improvement of the trivial upper bound for the height of the relevant determinant and consequently reduces the $4$-factor in the exponent to $2$.

\section{Background on algebraic number theory}\label{Backgrond}

Let $K$ be a number field and let $M_K$ denote the set of places of $K$ \cite[ p. 206]{Silverman}. For $v \in M_K$, we take the normalized representatives as follows. 

\noindent
If $v$ is an Archimedean place, then it is associated with an embedding 
$\sigma : K \to \mathbb{C}$ by the rule
$$|\alpha|_v := |\sigma(\alpha)|_\infty ,$$
where $|\cdot|_\infty$ denotes the usual real or complex absolute value, depending on whether $\sigma : K \to \mathbb{C}$ is a real or complex embedding.

\noindent
If $v$ is non-Archimedean, then it is associated with a prime ideal 
$\mathfrak{p}$ by the rule
$$|\alpha|_v := p^{-v_{\mathfrak{p}}(\alpha)/e_{\mathfrak{p}}} ,$$
where $v_{\mathfrak{p}}(\alpha)$ is the exponent of $\mathfrak{p}$ in the prime ideal factorization of the fractional ideal $\alpha \mathcal{O}_K$, and $e_{\mathfrak{p}}$ is the ramification index of $\mathfrak{p}$ over the rational prime $p$ \cite[p. 20]{Lang 2}. 

\noindent
For $v \in M_K$, let $K_v$ denote the completion of $K$ at $v$, and set
$$ n_v := [K_v : \mathbb{Q}_v].$$
It is well known that the product formula \cite[5.3]{Silverman}
\begin{equation}\label{product}
    \prod_{v \in M_K} |\alpha|_v^{n_v} = 1
\end{equation}
holds for all $\alpha \in K$, $\alpha \neq 0$.

\medskip
\noindent
For $\alpha \in K$, the (multiplicative) height \cite[p. 146]{Ellenberg} of $\alpha$ is defined by
\begin{equation}\label{hdef}
    H(\alpha) = \prod_{v \in M_K} \max \bigl(1, |\alpha|_v^{n_v} \bigr).
\end{equation}

\noindent
Observe that when $\alpha \in \KK{K}$, where $\KK{K}$ denotes the ring of integers of $K$, the height of $\alpha$ can be rewritten as
$$
H(\alpha)=\prod_{\sigma : K\to \mathbb{C}}\max\left(1, |\sigma(\alpha)|_\infty^{n_\sigma}\right).
$$
Here, $n_\sigma=1$ if $\sigma$ is a real embedding and $n_\sigma=2$ if $\sigma$ is a non-real complex embedding.

\vspace{0.1cm}
\noindent
Now, when $\alpha, \beta \in K$, it is well known that 
$$
H(\alpha+\beta)\leq 2^{[K:\mathbb{Q}]}H(\alpha)H(\beta),
$$
\noindent
which essentially reflects the multiplicative nature of the height function \cite[p. 146]{Ellenberg}. In Section \ref{Height Inequality}, we study the behavior of $H(\alpha+\beta)$ when both $\alpha$ and $\beta$ are algebraic integers, and show that for a large subset of elements $\alpha$, $\beta \in \KK{K}$ with bounded height, $H(\alpha+\beta)$ behaves additively.

\medskip

\noindent
For every algebraic number $\alpha \in K$, let us denote by 
$$
N(\alpha)=\prod_{i=1}^{n}\sigma_i(\alpha)
$$
the norm of $\alpha$ in $K$ over $\mathbb{Q}$, with $n=[K:\mathbb{Q}].$ Similarly, for each prime ideal $\mathfrak{p}$ in $\KK{K}$, we denote by 
$$
N\mathfrak{p}=|\KK{K}/\mathfrak{p}|=p^{\frac{n_v}{e_{\mathfrak{p}}}}
$$
the norm of the prime ideal $\mathfrak{p}$  \cite[p. 20]{Lang 2}.

\medskip
\noindent
Note that, for any $\alpha \in K^*$ one can show using the product formula (\ref{product}) that $H(\alpha)=H(\alpha^{-1})$. Now,

\begin{equation}\label{height to norm}
    H(\alpha) = H(\alpha^{-1})=\prod_{v \in M_K} \max \bigl(1, |\alpha|_v^{-n_v} \bigr) \geq \prod_{v_{\mathfrak{p}}(\alpha)> 0} p^{n_vv_{\mathfrak{p}}(\alpha)/e_{\mathfrak{p}}} = \prod_{v_{\mathfrak{p}}(\alpha)> 0} N\mathfrak{p}^{v_{\mathfrak{p}}(\alpha)}.
\end{equation}

\noindent
This relation is particularly useful for passing from the height to $\mathfrak{p}$-adic valuations. We will use this relation in the proof of Theorem \ref{L.S}.

\medskip
\noindent
To summarize, we have the following well-known properties \cite[p. 146]{Ellenberg} of the height function which we will use throughout the next sections.
For all $\alpha, \beta \in K^*$,
\begin{equation} \label{I2}
    H(\alpha) = H(\alpha^{-1}), 
    \qquad H(\alpha\beta) \leq H(\alpha) H(\beta), 
    \qquad H(\alpha+\beta) \leq 2^{[K:\mathbb{Q}]} H(\alpha) H(\beta).
\end{equation}

\medskip
\noindent
Before concluding this section, we fix some notations related to the unit group of $\KK{K}$, which will be used in Section \ref{Height Inequality}.

\medskip
\noindent
Let us denote by $\KK{K}^{*}$ the group of units of $\KK{K}$. By Dirichlet's unit theorem \cite[Theorem 100, Section 34]{Hecke}, we know that the unit group $\KK{K}^{*}$ is finitely generated (the generators are called fundamental units) with rank $r=r_1+r_2-1$, where $r_1$ is the number of real embeddings and $r_2$ is the number of conjugate pairs of complex embeddings. In particular,
$$
n=[K:\mathbb{Q}]=r_1+2r_2.
$$

\noindent
Now onwards, for simplicity by $|\sigma_i(\,\,\cdot\,\,)|,$ we will denote the usual complex absolute value of each embedding $\sigma_i$ (counting two conjugate embeddings defining the same place separately).
\section{An Additive Property of the Height Function}\label{Height Inequality}

In this section, our aim is to prove a much stronger additive height inequality than (\ref{2}) of the following type when both $\alpha$ and $\beta$ are algebraic integers coming from a large subset of algebraic integers up to height $H.$
\begin{proposition}\label{I1}
    Let $A\subseteq\{\alpha\in \KK{K}\ : H(\alpha)\leq H\}$. There exists  $A'\subseteq A$ with $|A'|\gg_K \frac{|A|}{(\log H)^r}$ such that for every $\alpha$ and $\beta \in A'$, we have 
    \begin{equation}\label{I0}
        H(\alpha+\beta)\ll_K (H(\alpha)+H(\beta)).
    \end{equation}
    
\end{proposition}

\noindent
In order to prove Proposition \ref{I1}, we will need the following lemma which says that the stronger height inequality holds if the order of the absolute values of two algebraic integers are preserved under every complex embedding.

\begin{lemma}\label{u1}
    Let $\alpha, \beta \in \KK{K}$, if $|\sigma_i(\alpha)|\leq c_i|\sigma_i(\beta)|$ for each  $\sigma_i:K\to \mathbb{C}$ and for some $c_i>0$, then $$H(\alpha+\beta)\leq c_K(H(\alpha)+H(\beta)),$$ 
    
    \noindent
    where $c_K$ is an explicit constant depending only on $c_i$ and the degree of $K$.
\end{lemma}

\begin{proof}
First observe that, 
$$H(\alpha +\beta)\leq\prod_{\sigma_i:K\to \mathbb{C}}\left(|\sigma_i(\alpha)|+|\sigma_i(\beta)|+1\right)\leq \prod_{\sigma_i:K\to \mathbb{C}}\left((c_i+1)|\sigma_i(\beta)|+1\right)< c\prod_{\sigma_i:K\to \mathbb{C}}\left(|\sigma_i(\beta)|+1\right),$$

\noindent
where $c=\prod_i(c_i+1)$. Note that, to get the first inequality we just use the triangle inequality and the fact that $(a+b) \geq \max\{a,b\}$, for $a,b\geq 0$. Also, since $\max\{a,b\} \geq \frac{(a+b)}{2}$, we have  
\begin{equation}\label{strict HI}
    H(\alpha+\beta)< c\prod_{\sigma_i:K\to \mathbb{C}}\left(|\sigma_i(\beta)|+1\right)< c.2^{[K:\mathbb{Q}]}H(\beta).
\end{equation}

\noindent
Hence we can conclude that $H(\alpha+\beta)<c_K(H(\alpha)+H(\beta)),$ for some constant $c_K$ only depending on the degree $[K:\mathbb{Q}]$ and the constants $c_i$. This estimate is sufficient for our purposes.
\end{proof}

\begin{remark}\label{remark}
    Note that in (\ref{strict HI}) we have a strict inequality. It turns out that we can produce a much sharper estimate when the constants $c_i$'s are larger than $1$ and $H(\alpha)$ is very small compared to $H(\beta).$ This can be done using a generalization of the Chebyshev's inequality \cite[Theorem 2, p. 37]{Inequality}. In fact, one can show that

$$H(\alpha+\beta)\leq c_K(H(\alpha)+H(\beta)),$$

\vspace{0.5cm}

\noindent
where $c_K=2^{2[K:\mathbb{Q}]-1}\max\{c,cc'\} $ with $c=\prod_i\max\{1,\frac{1}{c_i}\}$ and $c'=\prod_i\max\{c_i,1\}$.

\end{remark}

\noindent
Next, our aim is to show that for any set of algebraic integers up to height $H$, there is a large portion of elements in the set which satisfy the conditions of Lemma \ref{u1}. The following result from  \cite[p. 140-141]{Hecke} enables us to do so.

\medskip
\noindent

\begin{lemma}\label{u2}
   For every $\alpha \in K$ there exists a unit $u\in \KK{K}^*$ such that for every $1\leq i \leq n$,
   $$|\sigma_i(\alpha u)|\leq |N(\alpha)|^\frac{1}{n}e^{rM}$$
   where, $M$ denotes the absolute value of the numerically largest of the values $\log|\sigma_i(\epsilon_j)|$ for all $1\leq i\leq n$, $1\leq j\leq r$ and $\epsilon_1,...\epsilon_r$ denote $r$ fundamental units of $\KK{K}^*$
\end{lemma}

\begin{proof}
     First of all, note that for every non-zero $\alpha \in K$, there exist a uniquely determined system of real numbers
$c_1, c_2, \dots, c_r$ such that for the first $r$ conjugates we have (p. 140-141, \cite{Hecke})
\begin{equation} \label{lu}
\log \left|\frac{\sigma_i(\alpha)}{\sqrt[n]{N(\omega)}}\right|
=
c_1 \log \left| \sigma_i(\varepsilon_1) \right|
+ \cdots +
c_r \log \left| \sigma_i(\varepsilon_r) \right|
\qquad (i = 1,2,\dots,r).
\end{equation}

\noindent
Call the $c_i$ the exponents of $\alpha$.  Now since we have
\[
\sum_{i=1}^{r+1}
e_i
\log \left|\frac{\sigma_i(\alpha)}{\sqrt[n]{N(\omega)}}\right|
= 0
\quad \text{and} \quad
\sum_{i=1}^{r+1}
e_i
\log \left| \sigma_i(\varepsilon_k) \right|
= 0,
\]

\noindent
where $e_i=1$ if $\sigma_i$ is real and $e_i=2$ otherwise, 
equation (\ref{lu}) also holds for $i = r+1=r_1+r_2$ and consequently for all conjugates.
Next, by the Dirichlet's unit theorem \cite[Theorem 100]{Hecke} each unit $u' \in \KK{K}^*$ has the form
\[
u'=\zeta \, \varepsilon_1^{m_1} \varepsilon_2^{m_2}
\cdots
\varepsilon_r^{m_r},
\]
where $\zeta$ is one of the existing roots of unity in the field $K$,
while the $m_i \in \mathbb{Z}$, then clearly 
$\alpha u'$ has the exponents
\[
c_1 + m_1, \quad
c_2 + m_2, \quad \dots, \quad
c_r + m_r.
\]

\noindent
Consequently, for each $\alpha$ there is a unit $u \in \KK{K}^*$ such that the
exponents of $\alpha u$ satisfy the conditions
\[
0 \le c_i < 1
\qquad (i = 1,2,\dots,r).
\]

\noindent
As a consequence, we have
\[
\left| \sigma_i(\alpha u) \right|
=
\left| \sqrt[n]{N(\alpha)} \right|
e^{\left(
\sum_{j=1}^{r}
c_j \log \left| \sigma_i(\varepsilon_j) \right|
\right)}
\le
\sqrt[n]{N(\alpha) \,} \, e^{rM},
\qquad (i = 1,2,\dots,n).
\]
\noindent
where, $M$ denotes the absolute value of the numerically largest of the values $\log|\sigma_i(\epsilon_j)|$ for all $1\leq i\leq n$, $1\leq j\leq r$, which concludes the proof.
\end{proof}

\noindent
We also have the following useful Corollary of Lemma \ref{u2}.

\begin{corollary}\label{useful}
    For every $\alpha \in K$ there exists a unit $u\in \KK{K}^*$ such that for every $1\leq i \leq n$,
   $$|N(\alpha)|^\frac{1}{n}e^{(1-n)rM}\leq|\sigma_i(\alpha u)|\leq |N(\alpha)|^\frac{1}{n}e^{rM}.$$
   where, $M$ denotes the absolute value of the numerically largest of the values $\log|\sigma_i(\epsilon_j)|$ for all $1\leq i\leq n$, $1\leq j\leq r$ and $\epsilon_1,...\epsilon_r$ denote $r$ fundamental units of $\KK{K}^*$
\end{corollary}

\begin{proof}
\noindent
Applying Lemma \ref{u2} to each of the conjugates $\sigma_j(\alpha u)$ for $i\neq j$, we get

    $$|\sigma_i(\alpha u)|=\frac{|N(\alpha u)|}{|\prod_{j\neq i}\sigma_j(\alpha u)|}\geq \frac{|N(\alpha u)|}{|N(\alpha)|^{\frac{n}{n-1}}e^{(n-1)rM}}=|N(\alpha)|^\frac{1}{n}e^{(1-n)rM}.$$

    \noindent
Altogether, for any $\alpha \in K$ there exists a $u \in \KK{K}^*$ such that for every $1\leq i \leq n$, we have 
\begin{equation}\label{I3}
    |N(\alpha)|^\frac{1}{n}e^{(1-n)rM}\leq|\sigma_i(\alpha u)|\leq |N(\alpha)|^\frac{1}{n}e^{rM}.
\end{equation}
\end{proof}
\noindent
Let us now denote by $D$ the set of algebraic integers in $\KK{K}$ which satisfy (\ref{I3}).
\begin{equation}\label{Domain}
    D:= \{\alpha \in \KK{K}: |N(\alpha)|^\frac{1}{n}e^{(1-n)rM}\leq|\sigma_i(\alpha)|\leq |N(\alpha)|^\frac{1}{n}e^{rM} \text{ for each  }1\leq i\leq n \}.
\end{equation}

\noindent
Note that for any $\alpha \in D$ we have,
\begin{equation}
    H(\alpha)\leq \prod_i\left(|\sigma_i(\alpha)|+1\right) \leq |N(\alpha)|2^ne^{nrM}.\label{I4}
\end{equation}

\noindent
Now, by Lemma \ref{u2}, for every $\alpha \in A \subseteq \KK{K}$ there is a $u \in \KK{K}^*$ such that $\alpha u \in D$. Consequently, for any $\alpha \in A$ there is a $\beta \in D$ such that $\beta u' = \alpha \in A $ for some $u' \in \KK{K}^*.$ (Note that, $uu'=1).$

\medskip
\noindent
Also, using (\ref{I2}) and (\ref{I4}),
\begin{equation}\label{I5}
    H(u')=H(\beta^{-1}\alpha)\leq H(\beta)H(\alpha)\leq |N(\alpha u)|2^ne^{nrM}H=|N(\alpha)|2^ne^{nrM}H\leq H^22^ne^{nrM}.
\end{equation}

\noindent
Hence, there are only finitely many choices for $u'$. Moreover, the following result from \cite[Theorem 5.2]{Lang} gives an exact estimate on the number of choices for $u'$.

\begin{lemma}\label{u3}
   The number of units $u \in \KK{K}^*$ with $H(u)\leq H$ is given by 
   $$\gamma_K(\log H)^r+\mathcal{O}((\log H)^{r-1}),$$
   where, $\gamma_K$ is a constant only depending on the number field $K$.
\end{lemma}
\noindent
As an immediate consequence of Lemma (\ref{u3}), we have 
\begin{equation}\label{count unit}
    \#\{u \in \KK{K}^* : H(u)\leq H^22^ne^{nrM}\} \ll_K(\log H)^r.
\end{equation}
\noindent
So clearly, combining Corollary \ref{useful} and (\ref{I5}), we can write

\begin{equation}\label{I6}
    A \subseteq \bigcup_{j=1}^N (A \cap u_j D),
\end{equation}
where $u_j\in \KK{K}^*$ and $H(u_j)\leq H^22^ne^{nrM}$ for each $1 \leq j\leq N$ and using (\ref{count unit}) $N \ll_K(\log H)^r.$

\begin{proof}[Proof of Proposition \ref{I1}]
    First, write $A_j=A \cap u_j D.$ For any $\alpha u_j$ and $\beta u_j \in A_j$, we have 
    
    $$|\sigma_i(\alpha u_j)|\leq |\sigma_i(u_j)| |N(\alpha)|^\frac{1}{n}e^{rM}$$
    and 
    
    $$|\sigma_i(u_j)||N(\beta)|^\frac{1}{n}e^{(1-n)rM}\leq|\sigma_i(\beta u_j)|,$$
    for every $1\leq i \leq n.$\\
    
    \noindent
    Now, without loss of generality we can assume that $|N(\alpha)|\leq |N(\beta)|,$ which gives 

    $$|\sigma_i(\alpha u_j)|\leq e^{nrM}|\sigma_i(\beta u_j)|,$$
    for every $1\leq i \leq n.$

    \medskip
    \noindent
    Hence, by Lemma \ref{u1}

    $$H(\alpha u_j+\beta u_j) \leq C_{M,n,r} (H(\alpha u_j)+H(\beta u_j)),$$\\
    \noindent
    which implies that (\ref{I0}) holds for any two elements in $A_j$ for each $1\leq j \leq N$.

    \medskip
    \noindent
    Let us denote by $A_L$ the set of maximal cardinality among all the $A_j$'s. By (\ref{I6}) we have 
    $$|A|\ll_K(\log H)^r|A_L| ,$$
    \noindent
    which concludes the proof of Proposition \ref{I1}.

\end{proof}

\section{The Larger Sieve over $\KK{K}$}\label{Larger Sieve}

\begin{proof}[Proof of Theorem~\ref{L.S}]
Let 
$$\Delta = \prod_{\substack{a,b \in A' \\ a \neq b}} H(a - b),$$

\noindent
where $A'\subseteq A$  is the choice of the subset in Proposition \ref{I1}.

\noindent
Now, using Proposition \ref{I1}, we have the following upper bound
\begin{equation} \label{delta-upper}
    \Delta \leq (2C_K H)^{|A'| (|A'| - 1)},
\end{equation}
where $C_K$ is a constant only depending on the number field $K$.
Next, we bound the height from below using (\ref{height to norm}) as follows.  

$$ \Delta 
    =\prod_{\substack{a,b \in A' \\ a \neq b}} H(a - b)\geq\prod_{\substack{a,b \in A' \\ a \neq b}} 
        \prod_{\substack{\mathfrak{p} \in P \\ v_{\mathfrak{p}}(a-b) > 0}} 
        (N\mathfrak{p})^{v_{\mathfrak{p}}(a-b)}.$$
It follows that
\begin{align*}
    \log \Delta 
    &\geq 
    \sum_{\substack{a,b \in A' \\ a \neq b}} 
    \sum_{\substack{\mathfrak{p} \in P \\ v_{\mathfrak{p}}(a-b) > 0}} 
    \log N\mathfrak{p} \\
    &= 
    \sum_{\substack{a,b \in A' \\ a \neq b}} 
    \sum_{\substack{\mathfrak{p} \in P \\ a \equiv b \, (\mathrm{mod}\, \mathfrak{p})}} 
    \log N\mathfrak{p}.
\end{align*}

\noindent
For each $\mathfrak{p} \in P$ and $\mathfrak{u} \in A \,\,\, \mathrm{mod}\,\,\mathfrak{p} $, define
\[
    A'_{\mathfrak{p}}(\mathfrak{u}) = 
    \big| \{ a \in A' : a \equiv \mathfrak{u} \, (\mathrm{mod}\, \mathfrak{p}) \} \big|.
\]
Then
\begin{align*}
    \log \Delta 
    &\geq 
    \sum_{\mathfrak{p} \in P} 
    (\log N\mathfrak{p}) 
    \sum_{\substack{a,b \in A' \\ a \equiv b \, (\mathrm{mod}\, \mathfrak{p})}} 1
    - |A'| \sum_{\mathfrak{p} \in P} \log N\mathfrak{p} \\
    &= 
    \sum_{\mathfrak{p} \in P} 
    (\log N\mathfrak{p}) 
    \sum_{\mathfrak{u}\in A \,\,\, (\mathrm{mod}\,\,\mathfrak{p})} A'_{\mathfrak{p}}(\mathfrak{u})^2
    - |A'| \sum_{\mathfrak{p} \in P} \log N\mathfrak{p}.
\end{align*}

\noindent
By the Cauchy–Schwarz inequality, we have
\[
    \sum_{\mathfrak{u}\in A \,\,\, (\mathrm{mod}\,\,\mathfrak{p})} 
    A'_{\mathfrak{p}}(\mathfrak{u})^2 
    \geq 
    \frac{\big(\sum_{\mathfrak{u}} A'_{\mathfrak{p}}(\mathfrak{u})\big)^2}{\nu(\mathfrak{p})}
    = \frac{|A'|^2}{\nu(\mathfrak{p})}.
\]
Hence,
\[
    \log \Delta 
    \geq 
    \sum_{\mathfrak{p} \in P} 
    \Big( \frac{|A'|^2}{\nu(\mathfrak{p})} - |A'| \Big) \log N\mathfrak{p}.
\]

\noindent
Altogether with \eqref{delta-upper}, we obtain the desired inequality.
\[
    \sum_{\mathfrak{p} \in P} 
    \Big( \frac{|A'|^2}{\nu(\mathfrak{p})} - |A'| \Big) \log N\mathfrak{p}
    \leq 
    \log \Delta 
    \leq 
    |A'|(|A'| - 1) \log \big( 2C_K H \big).
\]

\end{proof}

\begin{proof}[Proof of Corollary~\ref{cor}]
Let $P$ be the set of prime ideals $\mathfrak{p}$ such that $N\mathfrak{p}\leq\frac{|A|}{(\log H)^r}$. Now, applying Theorem \ref{L.S} with this choice of $P$, we have 
$$ \frac{|A|}{(\log H)^r}\ll_{K,\alpha,c}H^{\alpha}.$$
\\
\noindent
So clearly, $|A|\ll_{K,\alpha}H^{\alpha}(\log H)^r,$ gives the required upper bound for $|A|.$
\end{proof}

\begin{remark}
    Note that, instead of working with the multiplicative height, if we use the height normalized by the degree of the number field (for such a height function, the height of an algebraic number does not depend on the choice of the number field), i.e. 
    $$\overline{H}(\alpha)={H(\alpha)}^{\frac{1}{[K:\mathbb{Q}]}} = {\prod_{v \in M_K} \max \bigl(1, |\alpha|_v^{n_v} \bigr)}^{\frac{1}{[K:\mathbb{Q}]}},$$
    \noindent
    we still have Theorem \ref{I1} and as a consequence Theorem \ref{L.S} and Corollary \ref{cor} follows.
\end{remark}

\section{Two-dimensional larger Sieve over $\KK{K}$}\label{2-D Larger}

In this section, our main aim is to prove a two-dimensional larger sieve of Helfgott and Venkatesh type over $\KK{K}$. One of the main ingredients of proving this two-dimensional larger sieve is a two-dimensional analogue of the improved height inequality (Proposition \ref{I1}).

\medskip
\noindent
Recall that, $\mathcal{W}\subseteq\{x^ly^m:0\leq l\leq L \,\,\,\text{and}\,\,\, 0\leq m\leq M\}$, where $M$ and $L$ are given.

\[|\mathcal{W}|=w, \,\,\,\,\,
d_{\mathcal{W}} = \sum_{f \in \mathcal{W}} \deg(f).
\]

\noindent
Let us also denote by $d_{\mathcal{W}x}$ and $d_{\mathcal{W}y}$ the total degrees of $x$ and $y$, respectively, among the monomials $f=x^ly^m\in\mathcal{W}$, where
\[
d_{\mathcal{W}x}=\sum_{x^ly^m\in\mathcal{W}}l,
\qquad
d_{\mathcal{W}y}=\sum_{x^ly^m\in\mathcal{W}}m.
\]
We write $f_1,f_2,...,f_w$ for the elements of $\mathcal{W}$ and by a $\mathcal{W}$-curve we mean an affine algebraic curve described by a single equation $g(x,y)=0$, where $g$ belongs to the linear span of $\mathcal{W}$ over $K$.

\noindent
First observe that for every $\alpha$ and $\beta \in D$ (\ref{Domain}), if we consider the quantity $\gamma=\prod_{f_j \in \mathcal{W'}} f_j(\alpha,\beta)$, where the product is over an arbitrary collection $\mathcal{W'}\subseteq\mathcal{W}$ of monomials, then
\begin{equation}\label{gen I}
    |N(\gamma)|^\frac{1}{n}\ll_{K,\mathcal{W}}|\sigma_i(\gamma)|\ll_{K,\mathcal{W}} |N(\gamma)|^\frac{1}{n} \text{ for each  }1\leq i\leq n.
\end{equation}

\noindent
The following lemma gives a simple generalization of Lemma \ref{u1}.

\begin{lemma}\label{F_2}
    Let $\alpha_t=\lambda \gamma_t$ where $\gamma_t$ satisfies (\ref{gen I}) for every $1\leq t \leq T$ and $\lambda \in K^*$. Then
    $$H(\sum_{t=1}^{T}\alpha_t)\ll_{K,\mathcal{W},T}\sum_{k=1}^{T}H(\alpha_t).$$ 
    \noindent
    
\end{lemma}

\begin{proof}
    Without loss of generality, we can assume that $N(\alpha_1)\leq N(\alpha_2)\leq...\leq N(\alpha_T),$ which essentially translates to saying 
    $$|\sigma_i(\alpha_1)|\ll_{K,\mathcal{W}}|\sigma_i(\alpha_2)|\ll_{K,\mathcal{W}}\,\,...\ll_{K,\mathcal{W}}|\sigma_i(\alpha_T)|, \text{ for each  } 1\leq i\leq n.$$
    \noindent
    Now  similarly as in the proof of Lemma \ref{u1}, we have 
    $$H(\sum_{t=1}^{T}\alpha_t)\ll_{K,\mathcal{W},T}\prod_{\sigma_i:K\to \mathbb{C}}\left(|\sigma_i(\alpha_T)|+1\right)\ll_{K,\mathcal{W},T}H(\alpha_T).$$
    which concludes the proof. \noindent
As mentioned in the Remark \ref{remark}, here also we can apply the generalization of the Chebyshev's inequality \cite[Theorem 2, p. 37]{Inequality} to obtain sharper constants in the inequality.
\end{proof}

\noindent
Next we have a two-dimensional analogue of Proposition \ref{I1}.

\begin{proposition}\label{H2}
    Let $B\subseteq\{(\alpha,\beta )\in \KK{K}\times \KK{K}\ : H(\alpha), H(\beta)\leq H\}$. There exists  $B'\subseteq B$ with $|B'|\gg_K \frac{|B|}{(\log H)^{2r}}$ such that for every $P_1, P_2,...,P_w \in B'$, we have 
    $$H(\text{det}((f_i(P_j))_{1\leq i,j\leq w}))\ll_{K,\mathcal{W}}H^{d_{\mathcal{W}}}.$$
\end{proposition}

\begin{proof}

For every $(x,y)\in \KK{K}\times \KK{K}$ and $(u,v)\in \KK{K}^*\times \KK{K}^*$, let us denote the point-wise multiplication by $(u,v)\cdot(x,y)=(ux,vy)$. As an immediate consequence of Lemma \ref{u2}, \ref{u3} and (\ref{I5}), similar to (\ref{I6}), we have
\begin{equation}\label{2-S}
    B \subseteq \bigcup_{1\leq i,j\leq N} (B \cap (u_i,u_j)\cdot (D\times D)),
\end{equation}
where, $u_i,u_j\in \KK{K}^*$ and $N \ll_K(\log H)^r$.\\

\noindent
For each $1\leq i,j\leq N$, let us denote $B \cap (u_i,u_j)\cdot (D\times D)$ by $B_{i,j} \subseteq B.$\\
\noindent
Observe that, for any $P_1, P_2,...,P_w \in B_{i,j},$ we can write 
$$\text{det}((f_i(P_j))_{1\leq i,j\leq w})=\sum_{\pi \in S_w} \operatorname{sgn}(\pi) f_1(P_{\pi(1)}) \cdots f_w(P_{\pi(w)})= \sum_{t}u_i^{d_{\mathcal{W}x}}v_j^{d_{\mathcal{W}y}}\gamma_t, $$

\noindent
where $\gamma_t$ satisfies (\ref{gen I}) for each $t$ and $H(u_i^{d_{\mathcal{W}x}}v_j^{d_{\mathcal{W}y}}\gamma_t)\ll_K H^{d_\mathcal{W}}.$\\

\noindent
Now, using Lemma \ref{F_2} we have $H(\text{det}((f_i(P_j))_{1\leq i,j\leq w}))\ll_{K,\mathcal{W}}H^{d_{\mathcal{W}}},$ for any $P_1, P_2,...,P_w \in B_{i,j}$. Let us assume that $B'$ is the set of maximal cardinality among all the $B_{i,j}$'s, then clearly 
$$|B|\ll_K(\log H)^{2r}|B'|.$$

\end{proof}
\noindent
We now proceed to the proof of Theorem \ref{2-D Larger Sieve}.

\begin{proof}[Proof of Theorem~\ref{2-D Larger Sieve}]

Let us assume that $(a)$ does not occur, in other words, every $\mathcal{W}$-curve contains at most $\delta|S'|$ points of $S'$ (with the choice of $S'$ coming from Proposition \ref{H2}).\\
\noindent
Write $\boldsymbol{\mathbf{P}}$ as shorthand for a
$w$-tuple $(P_1,\dots,P_w)$ of points in $S'$ and define 
$$W(\boldsymbol{\mathbf{P}})=W(P_1,...,P_w)=\text{det}((f_i(P_j))_{1\leq i,j\leq w}).$$
Note that, if $P_{j_1}=P_{j_2} \mod{\mathfrak{p}}$, then
$$f_i(P_{j_1})=f_i(P_{j_2}) \mod{\mathfrak{p}}, \,\,\,\text{for each }i.$$
As a consequence, for every prime ideal $\mathfrak{p}$, if the number of distinct points among the set $(P_1,...,P_w) \mod{\mathfrak{p}}$ is at most $k$, then the prime ideal factorization of $W(P_1,...,P_w)$ must contain $\mathfrak{p}^{w-k}$. We will use this property later.\\
\noindent
Let us now consider the following product over all tuples $\boldsymbol{\mathbf{P}} \in (S')^w$ with $W(\boldsymbol{\mathbf{P}})\neq 0$, which we shall call \emph{admissible}. Suppose,
\begin{equation}\label{eq:3.1}
\Delta \;=\; \prod\nolimits_{\boldsymbol{\mathbf{P}}}^{*}
 W(\boldsymbol{\mathbf{P}}).
\end{equation}

\noindent
Our aim is to estimate the quantity $H(\Delta)$ from above and below.\\

\noindent
Observe that, Proposition \ref{H2} immediately gives the following upper bound for all $\boldsymbol{\mathbf{P}} \in (S')^w$,

\[
H( W(\boldsymbol{\mathbf{P}}))\ll_{K,\mathcal{W}}H^{d_{\mathcal{W}}}.
\]

\noindent
Using (\ref{I2}) it follows that $H(\Delta)\leq {H( W(\boldsymbol{\mathbf{P}}))}^{|S'|^w}.$ Taking logarithm on both sides we get

\begin{equation}\label{eq:3.2}
\frac{\log H(\Delta)}{|S'|^{w}}
\le d_{\mathcal{W}}\log H + O_{K,\mathcal{W}}(1).
\end{equation}

\noindent
Next, we bound $H(\Delta)$ from below by a product of local terms.\\

\noindent
Fix a prime ideal $\mathfrak{p}$, with $N\mathfrak{p} \leq Q$, where $Q$ is a quantity that we set later for the purpose of our estimation.
For each $x\in (\KK{K}/\mathfrak{p})^{2}$, let $\rho_x$ be the fraction
of points in $S'$ that reduce to $x \bmod \mathfrak{p}$.
For each $\boldsymbol{\mathbf{P}}$, let
\[
\kappa(\boldsymbol{\mathbf{P}})\in \{0,1,\dots,w-1\}
\]
be such that $w-\kappa(\boldsymbol{\mathbf{P}})$ is the number of
distinct points among the $P_i \bmod \mathfrak{p}$.
Hence, we have the following bound (using the divisibility property of $W(\boldsymbol{\mathbf{P}})$ by $\mathfrak{p}^{w-k}$ as stated before)  for the $\mathfrak{p}$-valuation of $\Delta$ denoted by
$\operatorname{ord}_{\mathfrak{p}} \Delta$,
\begin{equation}\label{eq:3.3}
\operatorname{ord}_{\mathfrak{p}} \Delta
\;\ge\;
\sum\nolimits_{\boldsymbol{\mathbf{P}}}^{*}
\kappa(\boldsymbol{\mathbf{P}}),
\end{equation}
where the sum is only over all admissible
$\boldsymbol{\mathbf{P}}$.

\noindent
Now, using a probabilistic argument similar to  \cite[$3.4$]{H-V} we can write the sum 
$\sum_{\boldsymbol{\mathbf{P}}} \kappa(\boldsymbol{\mathbf{P}})$ taken
over all $\boldsymbol{\mathbf{P}}\in (S')^{w}$, admissible or not as the following.

\begin{equation}\label{eq:3.4}
\frac{1}{|S'|^{w}}
\sum_{\boldsymbol{\mathbf{P}}}
\kappa(\boldsymbol{\mathbf{P}})
=
w - \sum_{x\in (\KK{K}/\mathfrak{p})^{2}}
\bigl(1-(1-\rho_x)^{w}\bigr)
=
\sum_{x\in (\KK{K}/\mathfrak{p})^{2}}
\bigl((1-\rho_x)^{w} + w\rho_x -1\bigr).
\end{equation}

\noindent
It remains to estimate the sum of $\kappa(\boldsymbol{\mathbf{P}})$ over all non-admissible $\boldsymbol{\mathbf{P}}$.\\
\noindent
Consider the collection of all non-admissible tuples
$\boldsymbol{\mathbf{P}}$ for which $\kappa(\boldsymbol{\mathbf{P}})>0$.
For any such $\boldsymbol{\mathbf{P}}$, at least one of the following
situations must occur:

(1) There exist $(i,j)$ such that $P_i=P_j$;

(2) There exist $(i,j)$ such that
$P_i\equiv P_j \pmod{\mathfrak{p}}$ but $P_i\neq P_j$.\\
\noindent
Clearly, the number of tuples $\boldsymbol{\mathbf{P}}$ satisfying condition (1)
is bounded by $O_w(|S'|^{w-1})$.
To estimate the number of non-admissible tuples satisfying condition (2),
we permute the coordinates of $\boldsymbol{\mathbf{P}}$ so that $i=1$
and $j=2$, and reorder the elements of $W$ so that $w_1=1$ and
$w_2(P_1)\neq w_2(P_2)$.
(The first permutation introduces a factor of $w(w-1)/2$, which is absorbed
into the implied constant.)\\
\noindent
Now, observe that $\det(w_i(P_j))_{1\le i,j\le \ell}\neq 0$ for
$\ell=2$; let $\ell$ be maximal with this property.
Then $P_{\ell+1}$ lies on the $\mathcal{W}$-curve determined by
$P_1,P_2,\dots,P_\ell$ and by our assumption, such a curve contains at most $\delta|S'|$ points of $S'$. Hence there are at most $\delta|S'|$ possible values for $P_{\ell+1}$.\\
\noindent
It follows that the number of non-admissible tuples
$\boldsymbol{\mathbf{P}}$ satisfying condition (2) is
\[
O_w\!\left(\delta |S'|^{w-2}\Lambda\right),
\]
where $\Lambda$ denotes the number of pairs $(P,Q)\in (S')^2$ that reduce to
the same residue class modulo $\mathfrak{p}$, which can be expressed as the following by the definition of $\rho_x$.
\[
\Lambda = |S'|^2 \sum_x \rho_x^2.
\]

\noindent
Altogether, the total number of
non-admissible tuples $\boldsymbol{\mathbf{P}}$ with
$\kappa(\boldsymbol{\mathbf{P}})>0$ is at most
\[
|S'|^w \cdot O_w\!\left(|S'|^{-1}+\delta\sum_x \rho_x^2\right).
\]

\noindent
Using \eqref{eq:3.3} and \eqref{eq:3.4}, we have
\begin{equation}\label{eq:3.5}
\frac{\operatorname{ord}_{\mathfrak{p}} \Delta}{|S'|^w}
\ge
\sum_x\bigl((1-\rho_x)^w + w\rho_x -1\bigr)
-
O_w\!\left(\delta\sum_x \rho_x^2 + |S'|^{-1}\right).
\end{equation}

\noindent
Now, we the use the assumption that, for $N\mathfrak{p}>c$, the set $S'$ occupies at most $\tau |\KK{K}/\mathfrak{p}|=\tau N\mathfrak{p}$ residue classes modulo $\mathfrak{p}$ to produce a lower bound for the right side of (\ref{eq:3.5}).
The lower bound is obtained by considering the following two cases.\\

\noindent
\emph{Case 1.}
Suppose that for every $x\in(\KK{K}/\mathfrak{p})^{2}$ one has
\[
\rho_x < \frac{\delta}{w}.
\]
Then, for each $x$, we have the inequality
\[
(1-\rho_x)^w + w\rho_x -1
\;\ge\;
\biggl(\binom{w}{2}-O_w(\delta)\biggr)\rho_x^2.
\]
By Cauchy--Schwarz,
\[
\sum_x \rho_x^2
\;\ge\;
\frac{1}{\tau N\mathfrak{p}}\Bigl(\sum_x \rho_x\Bigr)^2
=
\frac{1}{\tau N\mathfrak{p}}.
\]
Substituting this into \eqref{eq:3.5}, we obtain
\begin{equation}\label{eq:3.6}
\frac{\operatorname{ord}_{\mathfrak{p}} \Delta}{|S'|^w}
\ge
\frac{1}{\tau N\mathfrak{p}}
\biggl(\binom{w}{2}-O_w(\delta)\biggr)
+
O_w(|S'|^{-1}).
\end{equation}

\medskip
\noindent
\emph{Case 2.}
Suppose instead that there exists some
$x\in(\KK{K}/\mathfrak{p})^{2}$ such that
\[
\rho_x \ge \frac{\delta}{w}.
\]
Since
\[
\frac{\partial}{\partial z}
\bigl((1-z)^w + wz -1\bigr)
=
w\bigl(1-(1-z)^{w-1}\bigr)
\ge wz,
\]
it follows that
\[
(1-\rho_x)^w + w\rho_x -1
\;\ge\;
\frac{1}{2}w\rho_x^2
\;\ge\;
\frac{1}{2}w\left(\frac{\delta}{w}\right)^2.
\]
Moreover, for $x'\neq x$ we have
$(1-\rho_{x'})^w + w\rho_{x'} -1 \ge 0$,
and clearly $\sum_x \rho_x^2 \le 1$.
Substituting into \eqref{eq:3.5} yields
\begin{equation}\label{eq:3.7}
\frac{\operatorname{ord}_{\mathfrak{p}} \Delta}{|S'|^w}
\ge
\frac{\delta^2}{2w}
-
O_w(\delta+|S'|^{-1}).
\end{equation}

\medskip
\noindent
Now, for prime ideals with  $N\mathfrak{p}$ exceeding a constant $c_{w,\delta}$ depending only on
$w$ and $\delta$, the bound \eqref{eq:3.7} implies \eqref{eq:3.6}.
(Here the implied constants in \eqref{eq:3.6} and \eqref{eq:3.7} need not
coincide.)
Henceforth, we shall work exclusively with the bound \eqref{eq:3.6}.

\noindent
Multiplying both sides of \eqref{eq:3.6} by $\log N\mathfrak{p}$ and summing over all prime ideal $\mathfrak{p}$ with $\max(c,c_{w,\delta})<N\mathfrak{p}\le Q$, we obtain
\[
\frac{w(w-1)}{2\tau}
\bigl(\log Q-\log c_{c,w,\delta}\bigr)
+
O_{\alpha,w}(\delta)
+
O(Q|S'|^{-1})
\;\le\;
\frac{\log (\prod_{\mathfrak{p}}N\mathfrak{p}^{\operatorname{ord}_{\mathfrak{p}} \Delta})}{|S'|^w}\leq \frac{\log H(\Delta)}{|S'|^w},
\]
where $c_{c,w,\delta}$ depends only on $c$, $w$, and $\delta$.

\medskip
\noindent
On the other hand, by \eqref{eq:3.2} we have
\[
\frac{\log H(\Delta)}{|S'|^{w}}
\le d_{\mathcal{W}}\log H + O_{K,\mathcal{W}}(1).
\]
Taking $Q=|S'|$, we finally arrive at
\[
\frac{w(w-1)}{2\tau}
\bigl(\log|S'|-\log c_{c,w,\delta}\bigr)
+
O_{\tau,w}(\delta)
\;\le\;
 d_{\mathcal{W}}\log H + O_{K,\mathcal{W}}(1).
\]

\noindent
Hence, 

$$\frac{|S|}{(\log H)^{2r}}\ll_{K}|S'| \;\ll_{c,\delta,w,\mathcal{W}}\;
H^{\frac{2d_{\mathcal{W}}\,\tau}{w(w-1)}+O_{\tau,\mathcal{W}}(\delta)}.
$$

$$\implies |S|\ll_{K,c,\delta,\mathcal{W}}H^{\frac{2d_{\mathcal{W}}\,\tau}{w(w-1)}+O_{\tau,\mathcal{W}}(\delta)+\epsilon}.$$

\medskip
\noindent
We can let $H$ to be sufficiently large so that $\epsilon$ is absorbed in the constant implied by $O_{\tau,\mathcal{W}}(\delta)$. This completes the proof.
    
\end{proof}

\section{Bombieri-Pila bound over number fields}\label{Bombieri-Pila}

In this section, we prove a Bombieri--Pila type bound over number fields by adapting the methods of \cite{H-V}.    Let $f(x,y) \in \mathcal{O}_K[x,y]$ is irreducible over $K$. We denote 
$$N(f,H)=\#\{(x,y)\in \mathcal{O}_K^2, H(x),H(y)\leq H : f(x,y)=0\} .$$

\noindent
Also, let us denote the maximum height of the coefficients of $f(x,y)$ by $||f||$. If $f$ is of degree $d$, we will show that $N(f,H)\ll_{K,d,\epsilon}H^{\frac{1}{d}+\epsilon}$. In order to do that, we will need Weil bound for curves over finite fields and a number field version of Heath-Brown's interpolation argument on curves. Let us denote by $\mathbb{F}_q$ the finite field with $q$ elements, here $q$ is a prime power. First we state the Weil bound for curves over finite feilds from \cite[p. 92]{Weil}.

\begin{theorem}[Weil bound]\label{Weil bound}
Let $g(x,y)\in \mathbb{F}_q[x,y]$ be an absolutely irreducible polynomial
of degree $d$. Then
\[
\#\{(x,y)\in \mathbb{F}_q^2 : g(x,y)=0\}
=
q + O_d(\sqrt{q}).
\]
\end{theorem}
\medskip

\noindent
Next, we need the following number field version of Heath-Brown's result from \cite{Heath-Brown}.

\begin{lemma}\label{Heath-Brown}
Let $f(x,y) \in \mathcal{O}_K[x,y]$ be irreducible over $K$ of degree $d$. Then either
$$N(f,H) \le d^2,$$
\noindent
or
$$f(x,y) = \lambda\, g(x,y),$$

\medskip
\noindent
where $\lambda \in K^*$ and $g(x,y) \in \mathcal{O}_K[x,y]$ satisfies $\|g\| \le H^{O_d(1)}$.
\end{lemma}

\begin{proof}
First, let us fix the following notations
\[
M := \frac{(d+1)(d+2)}{2}
\qquad\text{and}\qquad
N := d^2+1.
\]
Suppose that $N(f,H)\geq d^2+1$, i.e.
\[
f(x,y)=0
\]
admits $N$ solutions 
\[
(x_1,y_1),...,(x_N,y_N) \in \KK{K}^2,
\]
each satisfying $H(x_i), H(y_i) \leq H$.

\medskip

\noindent
We now construct an $N\times M$ integral matrix $C$ as follows.
The $i$-th row of $C$ consists of the of all the $M$ monomials of total
degree $\leq d$ in the variables $(x,y)$, evaluated at the point
$(x_i,y_i)$.

\medskip

\noindent
Let $f^*\in\KK{K}^M$ denote the coefficient vector of $f(x,y)$ with
respect to the monomial basis consisting of all the $M$ monomials of total degree $\leq d$.
Since each $(x_i,y_i)$ is a zero of $f$, we have
\[
Cf^*=0.
\]
As $f^*\neq 0$, it follows that the rank of $C$ is at most $M-1$.

\medskip
\noindent
Consequently, the homogeneous system
\[
Cg=0
\]
admits a nonzero solution $g^*\in\KK{K}^M$ which can be constructed explicitly from the $(M-1)\times(M-1)$
subdeterminants of $C$. Now, as $H(x_i), H(y_i)\leq H $ for all $i$, it is clear that for each entry $g^*_i$ of the vector $g^*$, we have 
\[
H(g^*_i)\leq H^{O_{d,M}(1)}=H^{O_d(1)}.
\]

\noindent
Let $g(x,y)$ be the polynomial of degree $\leq d$ whose coefficient
vector is $g^*$. By our construction, $g(x,y)$ vanishes at each of the points
$(x_1,y_1),...,(x_N,y_N) \in \KK{K}^2$.
Hence $f(x,y)$ and $g(x,y)$ have at least $d^2+1$ common zeros, which contradicts Bézout's theorem unless $g(x,y)$ is a constant multiple of $f(x,y)$ (as $f(x,y)$ is irreducible over $K$).

\medskip
\noindent
In particular, we conclude that
\[
f(x,y)= \lambda g(x,y)
\]
for some $\lambda \in K^*$ and $||g|| \leq H^{O_d(1)}$.
    
\end{proof}

\noindent
Now, we proceed to the proof of the main result of this section.

\begin{proof}[Proof of Theorem~\ref{B-P}]
    We may assume that
\[
N(f,H)=|S|>(d+1)^2,
\]
since otherwise the desired bound is trivial.

\medskip

\noindent
Now using Lemma \ref{Heath-Brown}, it is sufficient to restrict to polynomials satisfying $\|f\| \le H^{O_d(1)}$.

\medskip
\noindent
If the degree of $f$ in the variable $x$ is strictly less than $d$, we
apply a linear change of variables in $(x,y)$ to make the degree in $x$
equal to $d$.
Otherwise, we proceed without modification.
We choose such a linear transformation as follows.

\noindent
Let the linear transformation that makes the degree in $x$ equal to $d$ be given by
\begin{equation}\label{T1}
    (x,y)\longmapsto (a_1x+a_2y,a_3x+a_4y),
\end{equation}
where
\[
\max(H(a_1),H(a_2),H(a_3),H(a_4))=O_d(1).
\]

\noindent
To control the height of the integral points on the transformed curve, we require
\[
H(a_1x+a_2y) \ll_{K,d} H \quad \text{and} \quad H(a_3x+a_4y) \ll_{K,d} H.
\]
We achieve this by multiplying by units. By Corollary \ref{useful}, for every $(x,y) \in S$ there exist $(u_x,u_y) \in \KK{K}^* \times \KK{K}^*$ such that
\[
H(a_1u_xx+a_2u_yy) \ll_{K,d} H \quad \text{and} \quad H(a_3u_xx+a_4u_yy) \ll_{K,d} H.
\]
Hence, the transformation
\begin{equation}\label{T}
(x,y)\longmapsto (a_1u_xx+a_2u_yy,a_3u_xx+a_4u_yy)
\end{equation}
preserves the height of $(x,y)$ up to a constant depending only on $d$ and $K$. Furthermore, since (\ref{T1}) makes the degree in $x$ equal to $d$, the transformation (\ref{T}) only changes the coefficient of $x^d$ by a factor of $u_x^d$. Therefore, (\ref{T}) provides the required choice of linear transformation.

\medskip

\noindent
Moreover, by (\ref{I5}) and Lemma \ref{u3}, there are at most $\ll_K (\log H)^{2r}$ distinct pairs $(u_x,u_y)$. Hence, there are at most $\ll_K (\log H)^{2r}$ distinct transformations of the form \eqref{T}.

\medskip

\noindent
By the pigeonhole principle, there exists a transformation of the form \eqref{T} such that for at least
\[
\gg_K \frac{|S|}{(\log H)^{2r}}
\]
elements of $S$, the coordinates of the corresponding points on the transformed curve have height $\ll_{K,d} H$. It therefore suffices to establish the required bound for the transformed curve. From now on, we refer to the transformed curve as $C$.

\medskip
\noindent
Let $\mathfrak{p}$ be any prime ideal, and denote by $\bar f\in\mathbb{F}_q[x,y]$ the
reduction of $f$ modulo $\mathfrak{p}$, where $q=|\KK{K}/\mathfrak{p}|=N\mathfrak{p}$.
Write the factorisation of $\bar f$ into $\mathbb{F}_q$-irreducible
polynomials as
\[
\bar f=\bar f_1\cdots \bar f_{e_{\mathfrak{p}}}.
\]
By the Weil bound (Theorem \ref{Weil bound}), each irreducible component satisfies
\[
\bigl|\{(x,y)\in\mathbb{F}_q^2 : \bar f_i(x,y)=0\}\bigr|
\le q+O_d(\sqrt{q})=N\mathfrak{p}+O_d(\sqrt{N\mathfrak{p}}).
\]

\noindent
Let $\mathcal{P}$ denote the set of prime ideals for which $\bar f$ is reducible.
Let, $\Delta_f \in \KK{K}$ be the discriminant of $f$, then every prime ideal for which $\bar f$ is reducible, should appear in the prime ideal decomposition of the principle ideal generated by $\Delta_f$, i.e. $\mathfrak{p}|(\Delta_f)$.
Consequently,
\[
\prod_{\mathfrak{p}\in\mathcal{P}} N\mathfrak{p} \le H(\Delta_f) \leq H^{O_d(1)}.
\]

\noindent
For each $\mathfrak{p}\in\mathcal{P}$, partition the points of $S$ according to the irreducible component of $\bar f=0$ to which they reduce modulo $\mathfrak{p}$.
Since $\bar f$ has at most $d$ irreducible factors, this procedure yields a
covering of $S$ by sets
\[
S_1,\dots,S_k,
\]
with
\[
k\le d|\mathcal{P}| \ll_{d,\varepsilon} H^{\varepsilon},
\]
such that each $S_j$ intersects at most $N\mathfrak{p}+O_d(\sqrt{N\mathfrak{p}})$ residue classes
modulo every prime ideal $\mathfrak{p}$ .
In particular, for any fixed $\varepsilon>0$ and all prime ideal $\mathfrak{p}$, with
$N\mathfrak{p}\ge O_\varepsilon(1)$, each $S_j$ occupies at most
\[
(1+\varepsilon/2)N\mathfrak{p}
\]
residue classes modulo $\mathfrak{p}$.
\noindent
Hence it is enough to prove the desired bound with $S$ replaced by one
of the sets $S_j$.
Fix such a set and relabel it as $S$.

\noindent
We now apply Proposition \ref{2-D Larger Sieve} to this set $S$, taking
\[
\tau = 1+\varepsilon/2,
\]
with $L=d-1$. Since the curve $C$ is irreducible and has degree $d$ in $x$, whereas every
$\mathcal{W}$-curve has degree at most $d-1$ in $x$, the two curves cannot have a
common irreducible component. Indeed, any common component of $C$ and a
$\mathcal{W}$-curve would necessarily be equal to $C$, which is impossible since
the $\mathcal{W}$-curve has smaller degree in $x$. Therefore, by Bézout's theorem, the intersection of $S$ with any $\mathcal{W}$-curve contains at
most $d(d-1)$ points.
Hence, alternative~(a) in Proposition~3.1 would imply
\[
|S|\ll_K\delta^{-1}d(d-1)(\log H)^{2r}.
\]

\noindent
Assume instead that alternative~(b) of Proposition~3.1 holds.
Then we obtain the bound
\[
|S|
\ll_{K,\varepsilon,d,\delta,M}
H^{\frac{(1+\varepsilon/2)(d-1+M)}{d(M+1)-1}+
O_{\varepsilon,d,M}(\delta)}
.
\]
Choosing $M$ sufficiently large, this simplifies to
\[
|S|
\ll_{K,\varepsilon,d,\delta}
H^{\frac{1}{d}+3\varepsilon/4+O_{d,\varepsilon}(\delta)}.\]

\noindent
Finally, we choose $\delta$ small enough so that the term
$O_{d,\varepsilon}(\delta)$ is bounded by $\varepsilon/4$, and the desired
estimate follows.
In the remaining case corresponding to alternative~(a), we have
\[
|S|\ll_{K,d,\varepsilon}(\log H)^{2r},
\]
which is also admissible.

\end{proof}

\section*{Acknowledgments}
\noindent
I thank my supervisor Dr. Bryce Kerr for numerous fruitful discussions during the entire duration of this project and for his constant support and encouragement. I would also like to acknowledge the support of the Commonwealth through an Australian Government Research Training Program Scholarship.

\end{document}